\documentclass[12pt]{amsart}
\usepackage{amssymb}

\evensidemargin\paperwidth
\advance\evensidemargin-\textwidth
\advance\oddsidemargin-1in
\evensidemargin\oddsidemargin

\topmargin\paperheight
\advance\topmargin-\textheight
\advance\topmargin-1in

\theoremstyle{plain}
\newtheorem{theorem}{Theorem}[section]
\theoremstyle{plain}
\newtheorem{proposition}[theorem]{Proposition}
\theoremstyle{plain}
\newtheorem{lemma}[theorem]{Lemma}
\theoremstyle{plain}

\theoremstyle{plain}
\newtheorem{definition}[theorem]{Definition}
\theoremstyle{plain}

\theoremstyle{remark}

\theoremstyle{remark}

\theoremstyle{remark}

\title
[Property A and ONL]{Property A and the operator norm localization property for discrete metric spaces}
\author{Hiroki Sako}
\thanks{The author is a Research Fellow of the Japan Society for the Promotion of Science}
\address
{Research Institute for Mathematical Sciences, Kyoto University, Kyoto 606-8502, Japan}
\email
{sako@kurims.kyoto-u.ac.jp}

\begin{document}
\begin{abstract}
We study 
property A defined by G. Yu 
and the operator norm localization property 
defined by X. Chen, R. Tessera, X. Wang, and G. Yu.
These are coarse geometric properties for metric spaces, 
which have applications to operator K-theory.
It is proved that the two properties are equivalent 
for discrete metric spaces with bounded geometry.
\end{abstract}


\maketitle

\section{Introduction}
Coarse geometry is the study of large scale uniform structure on a space, 
which is related to operator K-theory.
In this paper
we investigate the following coarse geometric properties on metric spaces: property A 
and the operator norm localization property.
It is proved that these two properties 
are equivalent for metric spaces with bounded geometry.

\textit{Property A}
was introduced by G. Yu in \cite[Definition 2.1]{Yu:CoarseHilbert}.
A discrete metric space is said to have 
if the space satisfies a condition
regarding amenability.
Yu \cite{Yu:CoarseHilbert} proved that the property guarantees
the coarse Baum--Connes conjecture for metric space.
Most interesting case is 
when the metric space is given by a discrete group.
The Novikov higher signature conjecture holds for every discrete group $\Gamma$
with the property \cite[Theorem 1.1]{HigsonBivariant}. 
Property A for a discrete group can be characterized by exactness of
the reduced group C$^*$-algebra.
This follows from theorems by 
 Anantharaman-Delaroche and Renault \cite{ADR},
 Higson and Roe \cite{HigsonRoe},
 and Ozawa \cite{OzawaExactness}.

X. Chen, R. Tessera, X. Wang, and G. Yu introduced 
\textit{the operator norm localization property} in 
\cite[Definition 2.2 and Definition 2.3]{ONLPoriginal}.
They applied this notion to prove
that the coarse geometric Novikov conjecture holds
for certain sequences of expanders
\cite[Theorem 7.1]{ONLPoriginal}.

Property A and 
the operator norm localization property (ONL) have the following common features:
\begin{itemize}
\item
Finiteness of  asymptotic dimension implies these properties.
We refer to \cite{HigsonRoe} (property A) and \cite{ONLPoriginal} (ONL).

\item
If two groups have property A (resp.\ ONL), 
then an amalgamated free product has property A (resp.\ ONL).
We refer to \cite{BoundaryAme} (property A) and \cite{ONLPoriginal} (ONL). 

\item
If a group $\Gamma$ is a hyperbolic relative to subgroups with property A (resp.\ ONL), 
then $\Gamma$ has property A (resp.\ ONL).
We refer to \cite{BoundaryAme} (property A) and \cite{ONLPrelativehyp} (ONL). 

\item
Every countable subgroup of the general linear group over a field has these properties.
We refer to \cite{AforLinear} (property A) and \cite{ONLPlinear} (ONL).

\item
A sequence of expander graphs is an counterexample for these properties.
We refer to \cite{ONLPoriginal} (ONL).
\end{itemize}

A C$^*$-algebra $C^*_u(X)$ called uniform Roe algebra
is associated to a metric space $X$ with bounded geometry.
Skandalis, Tu, and Yu \cite[Theorem 5.3]{SkandalisTuYu} proved that $X$ has property A 
if and only if the uniform Roe algebra is nuclear
(see also Roe \cite[Proposition 11.41]{RoeLectureNote}             
and Brown--Ozawa \cite[Theorem 5.5.7]{Ozawa:Book}).    
Nuclearity can be characterized by a finite-dimensional approximation property
(Choi--Effros \cite{ChoiEffrossNuclearCPAP}, Kirchberg \cite{KirchbergNuclear}).
To obtain the equivalence between property A and the operator norm localization property,
we manipulate approximations by completely positive maps on the uniform Roe algebra.

In the last section, we make comments on a work by Brodzki, 
Niblo, \v{S}pakula, Willett, and Wright. By the main theorem in this paper and their result,
it turns out that two properties (MSP and ULA$_\mu$) are equivalent to property A.

\section{Preliminaries}
\subsection{Metric space with bounded geometry and uniform Roe algebra}
We fix some notations on a metric space $(X, d)$ and recall 
the definition of the uniform Roe algebra $C^*_u(X)$. 
For $S > 0$ and $x \in X$, we denote by $N(x, S)$ the closed ball
$\{y \in X \ |\ d(x, y) \le S\}$.
\begin{definition}
We say that $(X, d)$ has {\rm bounded geometry}, if
$\mathrm{sup}_{x \in X} |N(x, S)| < \infty$ for all $S > 0$.
\end{definition}
We note that a metric space with bounded geometry is a discrete space.
For a bounded linear operator $a$ on the Hilbert space $\ell_2(X)$,
we write $a_{y, z}$ for the matrix coefficient 
$\langle a \delta_z , \delta_y\rangle \in \mathbb{C}$.
The operator $a$ 
has the expression $a = [a_{y,z}]_{y,z \in X}$. 
We define 
the \textit{propagation} of $a$ by 
\begin{eqnarray*}
\mathrm{sup} \{ d(y, z) \ |\ y, z \in X, a_{y,z} \neq 0 \}.
\end{eqnarray*}
Let $E_R$ be the set of all the operators on 
$\ell_2(X)$ whose propagations are at most $R$.
The union $\cup E_R$ is a $*$-subalgebra of $\mathbb{B}(\ell_2(X))$. 
See the book \cite{RoeLectureNote} for details.
\begin{definition}
The C$^*$-algebra defined by the closure
$C^*_u(X) = \overline{\bigcup_{R > 0} E_R}^{\ \textrm{norm}}$
is called the {\rm uniform Roe algebra} of $X$.
\end{definition}

\subsection{Property A}
The definition of property A was given by G. Yu.
\begin{definition}[Definition 2.1 of \cite{Yu:CoarseHilbert}]
A discrete metric space $(X, d)$ is said to have {\rm property A} if for
every $R > 0$ and $\epsilon > 0$, there exist $S > 0$ and finite subsets 
$A_x \subseteq X \times \mathbb{N}$ labeled by $x \in X$ such that
\begin{itemize}
\item
if $d(y, z) \le R$, then $ | A_y \bigtriangleup A_z | < \epsilon | A_y \cap A_z |$,
where $A_y \bigtriangleup A_z$ stands for the symmetric difference of $A_y$ and $A_z$,
\item
and $A_x \subseteq N(x, S) \times \mathbb{N}$.
\end{itemize}
\end{definition}

Instead of the definition, we use the conditions
$(\ref{DefOfA;Vectors})$ and $(\ref{DefOfA;Kernel})$ in the next proposition.
\begin{proposition}
[Proposition 3.2 of \cite{Tu:CharacterizaitionOfA}]\label{DefOfA}
For a metric space $X$ with bounded geometry, the following conditions are equivalent:
\begin{enumerate}
\item\label{DefOfA;Original}
The space $X$ has property A,
\item\label{DefOfA;Vectors}
For every $\epsilon > 0$ and $R > 0$, there exist $S > 0$ 
and unit vectors $\{ \xi_x \}_{x \in X} \subseteq \ell_2(X)$
such that
\begin{itemize}
\item
${\rm if \ } d(y, z) \le R, {\rm \ then\ } \|\xi_y - \xi_z \| < \epsilon$,
\item
${\rm and } \ \mathrm{supp}(\xi_x) \subseteq N(x, S) {\rm \ for\ every\ } x \in X$,
\end{itemize}
\item\label{DefOfA;Kernel}
For every $\epsilon > 0$ and $R > 0$, there exist $S > 0$ 
and a positive definite kernel $k \colon X \times X \rightarrow \mathbb{C}$
such that  
\begin{itemize}
\item
$k(x, x) = 1 {\rm \ for\ every\ } x \in X$,
\item
${\rm if \ } d(y, z) \le R, {\rm \ then\ } |1 - k(y, z)| < \epsilon $,
\item
${\rm and \ if \ } d(y, z) > S, {\rm \ then\ }  k(y, z) = 0$.
\end{itemize}
\end{enumerate}
\end{proposition}
A function $k \colon X \times X \rightarrow \mathbb{C}$ is said to be
positive definite, if
for every $x(1)$, $\cdots$, $x(n)$ $\in X$ and 
$\lambda_1$, $\cdots$, $\lambda_n$ $\in \mathbb{C}$,
the inequality $\sum_{i, j =1}^{n} \overline{\lambda_i} \lambda_j k(x(i), x(j)) \ge 0$ 
holds true.

\subsection{The operator norm localization property}
X. Chen, R. Tessera, X. Wang, and G. Yu defined 
the operator norm localization property
in \cite{ONLPoriginal}. We call the property ``ONL'' in this paper.
The original definition is given for a general metric space $X$.
Let $\nu$ be a positive locally finite Borel measure on $X$ and 
let $\mathcal{H}$ be a separable infinite dimensional Hilbert space.
For an operator $b$ on the Hilbert space 
$L^2((X, \nu), \mathcal{H}) = L^2(X, \nu) \otimes \mathcal{H}$, 
the propagation of $b$ is also defined.
See section $2$ of \cite{ONLPoriginal} for details.
\begin{definition}[Definition 2.2 in \cite{ONLPoriginal}]
Let $\nu$ be a positive locally finite Borel measure on a metric space $X$. 
Let $f$ be a function $f \colon \mathbb{N} \rightarrow \mathbb{N}$.
We say that $(X, \nu)$ has ONL relative to $f$
with constant $0 < c \le 1$, if for every $R \in \mathbb{N}$ and every 
$a \in \mathbb{B}(L^2((X, \nu), \mathcal{H}))$ whose propagation is at most $R$, 
there exists a non-zero vector $\zeta \in L^2((X, \nu), \mathcal{H})$
satisfying $\mathrm{diam}(\mathrm{supp}(\zeta)) \le f(R)$
and $c \|a \| \| \zeta \| \le \|a \zeta \|$.
\end{definition}
\begin{definition}[Definition 2.3 in \cite{ONLPoriginal}]
A metric space $X$ is said to have ONL
if there exists a constant $0 < c \le 1$ and a function 
$f \colon \mathbb{N} \rightarrow \mathbb{N}$ such that
for every positive locally finite Borel measure $\nu$ on $X$, $(X, \nu)$
has ONL relative to $f$
with constant $c$.
\end{definition}
For the rest of this paper, we will concentrate on a metric space $X$ with bounded geometry.
By proposition \ref{properties}, we have only to consider 
the case that $\nu$ is the counting measure on $X$.
An operator $a$ on the Hilbert space 
$\ell_2(X, \mathcal{H}) = \ell_2(X) \otimes \mathcal{H}$
has the expression $a = [a_{y,z}]_{y,z \in X}$, 
where $a_{y,z}$ is an operator on $\mathcal{H}$.
The propagation of $b$ is equal to 
$\mathrm{sup} \{ d(y, z) \ |\ y, z \in X, a_{y,z} \neq 0 \}$.
Denote by $E_R(\mathcal{H})$ the set of the operators on 
$\ell_2(X, \mathcal{H})$ whose propagations are at most $R$.

\section{Characterizations of ONL}
We rephrase the definition of ONL for metric spaces with bounded geometry.
We need a few more notations to state the next proposition.
Denote by $B_S$ the C$^\ast$-algebra 
$\prod_{x \in X} \mathbb{B}(\ell_2(N(x , S)))$,
which is isomorphic to a product of matrix algebras.
An element $b \in B_S$
is a family of matrices $([b^{(x)}_{y,z}]_{y, z \in N(x, S)})_{x \in X}$
labeled by $X$.
For $S > 0$, define a linear map $\psi_S \colon C^\ast _u(X) \rightarrow B_S$ by 
$\psi_S(a) = ([a_{y,z}]_{y,z \in N(x, S)})_{x \in X}$.

\begin{proposition}\label{properties}
Let $X$ be a metric space with bounded geometry. The following
properties on $X$ are equivalent:
\begin{enumerate}
\item\label{onlp}
The space $X$ has ONL,
\item\label{ampexistence}
There exists $0 < c \le 1$ such that for every $R > 0$,
there exists $S > 0$ satisfying the condition
$(\alpha)$:
for every operator $a \in E_R(\mathcal{H})$,
there exists a non-zero vector $\zeta \in \ell_2(X, \mathcal{H})$ such that
$\mathrm{diam} (\mathrm{supp}(\zeta)) \le S$ and
$c \| a \| \| \zeta \| \le \| a \zeta \| ,$
\item\label{ampforall}
For every $0 < c < 1$ and $R > 0$,
there exists $S > 0$ satisfying
$(\alpha)$,
\item\label{withoutamp}
For every $0 < c < 1$ and $R > 0$,
there exists $S > 0$ satisfying the condition
$(\beta)$:
for every operator $a \in E_R$,
there exists a non-zero vector $\xi \in \ell_2(X)$ such that
$\mathrm{diam} (\mathrm{supp}(\xi)) \le S$ and
$c \| a \| \| \xi \| \le \| a \xi \|$,
\item\label{onlpbynorm}
For every $\epsilon > 0$ and $R > 0$,
there exists $S > R$ satisfying
\begin{eqnarray*}
\left\| (\psi_S |_{E_R})^{-1} \colon \psi_S(E_R) \rightarrow E_R \right\| < 1 + \epsilon.
\end{eqnarray*}
\end{enumerate}
\end{proposition}
In the next section, 
we will make use of property $(\ref{onlpbynorm})$.

\begin{proof}[Proof of $(\ref{onlp}) \Rightarrow (\ref{ampexistence})$]
If $X$ has ONL,
then $X$ satisfies ONL for the counting measure. This is nothing
but the property $(\ref{ampexistence})$.
\end{proof}

\begin{proof}[Proof of $(\ref{ampexistence}) \Rightarrow (\ref{onlp})$]
Suppose that $X$ has the property $(\ref{ampexistence})$ for a constant $c$.
For an arbitrary $R \in \mathbb{N}$, 
there exists $S = f(R)$ which satisfies $(\alpha)$.
We may choose $S$ from $\mathbb{N}$.
Observe that for every positive measure $\nu$,
the Hilbert space $L^2((X, \nu), \mathcal{H})$ can be identified with 
a closed subspace of 
$\ell_2(X, \mathcal{H}) = \ell_2(X) \otimes \mathcal{H}$.
The inclusion map is
\begin{eqnarray*}
V \colon L^2((X, \nu), \mathcal{H}) \ni \eta 
\longmapsto \Sigma_{x \in X} \nu(\{x\})^{1/2}  \delta_x \otimes \eta(x) \in
\ell_2(X) \otimes \mathcal{H}.
\end{eqnarray*}
Let $a$ be an arbitrary operator on $L^2((X, \nu), \mathcal{H})$ 
whose propagation is at most $R$.
Then the propagation of $V a V^*$ is at most $R$.
By the condition $(\alpha)$, 
there exists a non-zero vector $\zeta \in \ell_2(X, \mathcal{H})$ such that
\begin{eqnarray*}
\mathrm{diam} (\mathrm{supp}(\zeta)) \le S, \quad
c \| V a V^* \| \| \zeta \| \le \| V a V^* \zeta \|.
\end{eqnarray*}
These inequalities imply 
$\mathrm{diam} (\mathrm{supp}(V^* \zeta)) \le S$ and 
$c \| a \| \| V^* \zeta \| \le \| a V^* \zeta \|$.
We conclude that $X$ has ONL.
\end{proof}

The part ``there exists $0 < c \le 1$'' in $(\ref{ampexistence})$ 
can be replaced by ``for every $0 < c < 1$'' in $(\ref{ampforall})$.
Inspired by an idea of \cite[Proposition 2.4]{ONLPoriginal},
we give a complete proof.

\begin{proof}[Proof of $(\ref{ampexistence}) \Rightarrow (\ref{ampforall})$]
Assume that $X$ has the property $(\ref{ampexistence})$ 
with respect to a constant $0 < c < 1$.
For every $R$ and $n \in \mathbb{N}$, 
there exists $S$ satisfying the condition $(\alpha)$ for propagation $2nR$. 
Let $a \in E_R(\mathcal{H})$ be an arbitrary operator of norm $1$.
Since the propagation of $(a a^*)^n$ is at most $2nR$,
there exists a non-zero vector $\zeta \in \ell_2(X, \mathcal{H})$
such that $\mathrm{diam} (\mathrm{supp}(\zeta)) \le S$ and
$c \| (a a^*)^n \| \| \zeta \| \le \| (a a^*)^n \zeta \|$.
Since the norm of $(a a^*)^n$ is $1$, we have
\begin{eqnarray*}
c \le \prod_{j = 0}^{n-1} {\| (a a^*)^{j + 1} \zeta \|} / {\| (a a^*)^{j} \zeta \|}.
\end{eqnarray*}
It follows that there exists $j = 0, 1, \cdots, n - 1$ such that
$c^{1/n} \le \| (a a^*)^{j + 1} \zeta \| / \| (a a^*)^{j} \zeta \| $.
We have the inequality
\begin{eqnarray*}
c^{1/n} \| a \| \| a^* (a a^*)^{j} \zeta \| \le c^{1/n} \| (a a^*)^{j} \zeta \| 
\le \| a (a^* (a a^*)^{j} \zeta) \|.
\end{eqnarray*}
The diameter of $\mathrm{supp}(a^* (a a^*)^{j} \zeta)$ is at most $(2n - 1) R + S$.
It follows that the condition $(\alpha)$ holds for $c^{1/n}$, $R > 0$, and $(2n - 1) R + S$.
We can make the constant $0 < c^{1/n} < 1$ arbitrarily close to $1$, by enlarging $n$.
Hence the property $(\ref{ampforall})$ holds.
\end{proof}

The implication from $(\ref{ampforall})$ to $(\ref{ampexistence})$ is trivial.
We further rephrase the property.
The equivalence between $(\ref{ampforall})$ and $(\ref{withoutamp})$ 
means that the amplification by $\mathcal{H}$ is not necessary.

\begin{proof}[Proof of $(\ref{ampforall}) \Rightarrow (\ref{withoutamp})$]
Assume that $0 < c < 1$, $R > 0$, and $S > 0$ 
satisfy the condition $(\alpha)$. Fix a unit vector $\eta \in \mathcal{H}$. 
Denote by $e$ the rank one projection onto 
$\mathbb{C} \eta \subseteq \mathcal{H}$.
Let $a$ be an arbitrary operator on $\ell_2(X)$ whose propagation is at most $R$.
Since the propagation of $a \otimes e$ is at most $R$, 
there exists a non-zero vector $\zeta \in \ell_2(X) \otimes \mathcal{H}$ 
such that $\mathrm{diam} (\mathrm{supp}(\zeta)) \le S$ and
$c \| a \otimes e\| \| \zeta \| \le \|(a \otimes e) \zeta\|$.
The vector $(1 \otimes e) \zeta$ is of the form 
$\xi \otimes \eta \in \ell_2(X) \otimes \mathcal{H}$.
We have
\begin{eqnarray*}
&\mathrm{diam} (\mathrm{supp}(\xi)) \le 
\mathrm{diam} (\mathrm{supp}(\zeta)) \le S,&\\
&c \| a \| \| \xi \|
\le c \| a \otimes e\| \|\zeta\| 
\le \| (a \otimes e)  \zeta \| 
= \| a \xi \|.&
\end{eqnarray*}
It follows that $0 < c < 1$, $R > 0$, and $S > 0$ 
satisfy the condition $(\beta)$.
We conclude that $X$ satisfies the property $(\ref{withoutamp})$.
\end{proof}

\begin{proof}[Proof of $(\ref{withoutamp}) \Rightarrow (\ref{ampforall})$]
We assume that the condition $(\beta)$ holds for 
$0 < 1 - \epsilon/2 < 1$, $R$, and $S$.
Let $a$ be an operator on $\ell_2(X) \otimes \mathcal{H}$ whose propagation is at most $R$.

We claim that there exist isometries
$V, W \colon \ell_2(X) \rightarrow \ell_2(X) \otimes \mathcal{H}$ satisfying that
$V \delta_x, W \delta_x \in \mathbb{C} \delta_x \otimes \mathcal{H}$ and
$(1 - \epsilon/2) \| a \| \le \| W^* a V \| \le \| a \| $.
Take unit vectors $\zeta_1, \zeta_2 \in \ell_2(X) \otimes \mathcal{H}$ 
such that $(1 - \epsilon/2) \| a \| \le \langle a \zeta_1, \zeta_2 \rangle$.
The vectors $\zeta_1$, $\zeta_2$ can be written in the following forms:
$\zeta_1 = \Sigma_{x \in X} f(x) \delta_x \otimes \eta_1(x)$, 
$\zeta_2 = \Sigma_{x \in X} g(x) \delta_x \otimes \eta_2(x)$, where
$\eta_1(x)$, $\eta_2(x)$ are unit vectors
and $f(x), g(x) \in \mathbb{C}$.
We define two isometries $V, W\colon \ell_2(X) \rightarrow \ell_2(X) \otimes \mathcal{H}$ by
$V \delta_x = \delta_x \otimes \eta_1(x)$, 
$W \delta_x  = \delta_x \otimes \eta_2(x)$.
Since the vectors $\zeta_1$, $\zeta_2$ are respectively in the images of $V$, $W$,
the operator norm of $W^* a V$ satisfies 
$(1 - \epsilon/2) \| a \| \le \|W^* a V\|$.
Here we get the claim.

By the condition $(\beta)$, there exists a unit vector $\xi \in \ell_2(X)$ satisfying
\begin{eqnarray*}
\mathrm{diam} (\mathrm{supp}(\xi)) \le S, \quad
(1 - \epsilon/2) \| W^* a V \| \le \| W^* a V \xi \|.
\end{eqnarray*} 
Since the support of $V \xi$ is equal to that of $\xi$,
we have $\mathrm{diam} (\mathrm{supp}(V \xi)) \le S$.
We also obtain the inequality
\begin{eqnarray*}
(1-\epsilon) \|a\| \|V \xi\|
&=& (1-\epsilon) \|a\|
\le (1 - \epsilon/2)^2 \| a \| \\
&\le& (1 - \epsilon/2) \| W^* a V \| 
\le \| W^* a V \xi \| \\
&\le& \|a V \xi\|.
\end{eqnarray*} 
We get the condition $(\beta)$ for 
$0 < 1 - \epsilon < 1$, $R > 0$, and $S > 0$.
It follows that $X$ has the property $(\ref{ampforall})$.
\end{proof}

\begin{proof}[Proof of $(\ref{withoutamp}) \Rightarrow (\ref{onlpbynorm})$]
Assume that $X$ has the property $(\ref{withoutamp})$. 
For arbitrary $\epsilon > 0$ and $R > 0$, 
there exists $S$ which satisfies the condition $(\beta)$ 
for $c = (1 + \epsilon)^{-1}$ and $R > 0$.

It follows that for every non-zero operator $a \in E_R$, 
there exists a unit vector $\xi \in \ell_2(X)$ satisfying
$\mathrm{diam} (\mathrm{supp}(\xi)) \le S$ and
$\| a \| \le (1 + \epsilon) \| a \xi \|.$
Since the propagation of $a$ is at most $R$, 
the diameter of $\mathrm{supp}(a \xi)$ is included in
the $R$-neighborhood of $\mathrm{supp}(\xi)$.
Hence there exists a unit vector $\eta$ such that 
$\| a \xi \| = \langle a \xi, \eta \rangle$ and that 
supports of $\xi$, $\eta$ are included in a common
closed ball $N(x, S + R)$.
By the inequality
\begin{eqnarray*}
\| a \| \le (1 + \epsilon) \langle a \xi, \eta \rangle
\le (1 + \epsilon) \| [a_{y,z}]_{y, z \in N(x, S + R)} \| \le (1 + \epsilon) \| \psi_{S +R} (a) \|,
\end{eqnarray*}
we get $\left\| (\psi_{S + R} |_{E_R})^{-1} \right\| \le 1 + \epsilon$.
\end{proof}

\begin{proof}[Proof of $(\ref{onlpbynorm}) \Rightarrow (\ref{withoutamp})$]
Suppose that the property $(\ref{onlpbynorm})$ holds true. 
For every $0 < c <1$ and $R > 0$,
take $S$ satisfying 
$\| (\psi_{S} |_{E_R})^{-1} \colon \psi_S(E_R) \rightarrow E_R \| < c^{-1}$.
Then for every operator $a \in E_R$, 
there exists a closed ball $N(x, S)$ with radius $S$ satisfying
\begin{eqnarray*}
c \| a \| \le \| [a_{y, z}]_{y, z \in N(x, S)} \|.
\end{eqnarray*}
Take a unit vector $\xi \in \ell_2(N(x, S))$ such that
$\| [a_{y, z}]_{y, z \in N(x, S)} \| = \| [a_{y, z}]_{y, z \in N(x, S)} \xi \|$.
The vector $\xi$ satisfies
$\mathrm{diam} (\mathrm{supp}(\xi)) \le 2S$ and $c \| a \| \le \| a \xi \|$.
It follows that the condition $(\beta)$
holds true for $0 < c <1$, $R$, and $2S$.
We conclude that $X$ has the property $(\ref{onlpbynorm})$.
\end{proof}

For the proof of Theorem \ref{main}, 
we recall the notions of completely positive map and completely bounded map.
\begin{itemize}
\item
A self-adjoint closed subspace $F$ of a unital C$^\ast$-algebra $B$
such that $1_B \in F$ is called an \textit{operator system}.
\item
A linear map $\phi$ from $F$ to a C$^\ast$-algebra $C$
is said to be \textit{completely positive} if the map
$\phi^{(n)} = \phi \otimes \mathrm{id} 
\colon F \otimes \mathbb{M}_n(\mathbb{C}) \rightarrow C \otimes \mathbb{M}_n(\mathbb{C})$
is positive for every $n$.
\end{itemize}
The subspaces $E_R \subseteq C^*_u(X)$ 
and $\psi_S(E_R) \subseteq B_S$ are examples of operator systems.
The map $\psi_S \colon E_R \rightarrow B_S$ is an example of a completely positive map. 
\begin{itemize}
\item
A linear map $\theta \colon F \rightarrow C$ is said to be \textit{completely bounded} 
if the increasing sequence $\{\|\theta^{(n)} \colon 
F \otimes \mathbb{M}_n(\mathbb{C}) 
\rightarrow C \otimes \mathbb{M}_n(\mathbb{C}) \|\}$ is bounded.
The number $\| \theta \|_{\rm cb} = \mathrm{sup}_{n \in \mathbb{N}} \| \theta^{(n)}\|$ 
is called the \textit{completely bounded norm} of $\theta$.
\end{itemize}
The norms $\| \theta \|$ 
and $\| \theta \|_{\rm cb}$ are not identical in general, but
we have the following proposition.
\begin{proposition}\label{Norm}
For every $S, R$ such that $0 < R < S$, the completely bounded norm
$\| (\psi_S |_{E_R})^{-1}  \colon \psi_S(E_R) \rightarrow E_R \|_{\rm cb} $
 is equal to $\| (\psi_S |_{E_R})^{-1}\|$.
\end{proposition}

\begin{proof}
It suffices to show that  
$\| ((\psi_S |_{E_R})^{-1})^{(n)} \| \le \| (\psi_S |_{E_R})^{-1} \|$
for every $n \in \mathbb{N}$.
Take an arbitrary positive number $K$ satisfying $K < \| ((\psi_S |_{E_R})^{-1})^{(n)} \|$.
There exists an operator $a \in E_R \otimes \mathbb{M}_n(\mathbb{C})$
satisfying $K < \| a \| $ and $\| \psi_S^{(n)}(a) \| = 1$.
We claim that there exist isometries
$V, W \colon \ell_2(X) \rightarrow \ell_2(X) \otimes \mathbb{C}^n$ satisfying that
$V \delta_x, W \delta_x \in \mathbb{C}\delta_x \otimes \mathbb{C}^n$ and
$K < \| W^* a V \| \le \| a \| $.
Indeed, the proof of
$ (\ref{withoutamp}) \Rightarrow (\ref{ampforall})$ in Proposition \ref{properties} 
also works, 
$\mathcal{H}$ being replaced by $\mathbb{C}^n$.
Observe that the propagation of $W^* a V$ is at most $R$ and
that $\| \psi_S(W^* a V) \| \le \| \psi_S^{(n)}(a) \| = 1$.
It follows that $K < \| W^* a V \| / \| \psi_S(W^* a V) \| \le \| (\psi_S |_{E_R} )^{-1} \|$.
We conclude that $\| ((\psi_S |_{E_R})^{-1})^{(n)} \| \le \| (\psi_S |_{E_R})^{-1} \|$.
\end{proof}

\section{Main theorem}

\begin{theorem}\label{main}
Let $X$ be a metric space with bounded geometry.
The space $X$ has property A if and only if
$X$ has ONL.
\end{theorem}

Before proving Theorem \ref{main},
we recall a lemma in the book \cite{RoeLectureNote}.
This lemma allows us to bound operator norms by matrix coefficients.

\begin{lemma}[Lemma 4.27 in \cite{RoeLectureNote}]\label{NormConstant}
Let $X$ be a metric space with bounded geometry.
For each $R > 0$, there exists a constant $\kappa = \kappa(X, R)$ such that 
$\|a\| \le \kappa\ \mathrm{sup}_{y, z \in X} |a_{y, z}|$ holds for any $a \in E_R$.
\end{lemma}

We need a completely positive perturbation of the completely bounded map 
$(\psi_S|_{E_R})^{-1} \colon \psi_S(E_R) \rightarrow E_R \hookrightarrow  \mathbb{B}(\ell_2(X))$.
The following is Corollary B.8 in \cite{Ozawa:Book}:\\
Let $F \subseteq B$ be an operator system of a C$^*$-algebra $B$
and $\theta \colon F \rightarrow \mathbb{B}(\mathcal{H})$
be a unital self-adjoint map. Then, there exists a unital completely positive map
$\phi\colon B \rightarrow \mathbb{B}(\mathcal{H})$ 
satisfying $\| \theta - \phi |_F \|_{\rm cb} \le 2(\| \theta \|_{\rm cb} -1)$.\\
The statement is slightly modified, but this is what the proof actually shows.

\begin{proof}[Proof of Theorem \ref{main}]
We first assume that $X$ has property A.
Take arbitrary $R > 0$ and $\epsilon > 0$.
Let $\kappa = \kappa(X, R)$ is the constant given in Lemma \ref{NormConstant}.
By the condition (\ref{DefOfA;Vectors}) of Proposition \ref{DefOfA}, there exist unit vectors 
$\{\xi_x\}_{x \in X} \subseteq \ell_2(X)$ and a positive number $S$ satisfying
the following:
\begin{eqnarray*}
\mathrm{supp} (\xi_x) \subseteq N(x, S), \quad
| 1 - \langle \xi_y , \xi_z\rangle | 
< \epsilon/ \kappa \quad  (d(y, z) \le R).
\end{eqnarray*}
Define a linear map $\phi \colon B_S \rightarrow C^\ast_u (X)$ by 
\begin{eqnarray*}
\phi \left(\left([b^{(x)}_{y,z}]_{y, z \in N(x, S)} \right)_{x \in X} \right) 
= 
\left[ \sum_{x \in X} \xi_y(x) \overline{\xi_z(x)} b^{(x)}_{y,z} \right]_{y, z \in X}.
\end{eqnarray*}
We note that $\phi$ is unital and completely positive.

For $a \in C^\ast _u(X)$, we have 
\begin{eqnarray*}
\phi \circ \psi_S (a) 
&=& \phi \left( \left( [a_{y,z}]_{y,z \in N(x, S)} \right)_{x \in X} \right) \\
&=& \left[\sum_{x \in X} \xi_y(x) \overline{\xi_z(x)} a_{y,z} \right]_{y, z \in X} \\ 
&=& \left[ \langle \xi_y, \xi_z \rangle a_{y,z} \right]_{y, z \in X}.
\end{eqnarray*}
For $a \in E_R$, 
we get
\begin{eqnarray*}
\|a - \phi \circ \psi_S (a)\| 
&\le &  \kappa 
\left( \mathrm{sup}_{y, z \in X} |(1 - \langle \xi_y, \xi_z \rangle) a_{y,z}| \right)\\
&\le & \epsilon\ \mathrm{sup}_{y, z \in X} | a_{y,z}|\\
&\le & \epsilon\ \| a \|.
\end{eqnarray*}
This implies the following inequalities:
\begin{eqnarray*}
&\| a \|
\le  \| \phi \circ \psi_S (a)\| + \| a - \phi \circ \psi_S (a)\| 
\le  \| \psi_S (a)\| + \epsilon \|a\|,& \\
&(1 -\epsilon) \| a \| 
\le  \| \psi_S (a)\|.&
\end{eqnarray*}
It follows that the property $(\ref{onlpbynorm})$ in Proposition \ref{properties} holds true.
Hence $X$ has ONL.

Now assume that $X$ has ONL.
By Proposition \ref{properties} $(\ref{onlpbynorm})$ and Proposition \ref{Norm}, 
for any $R > 0$ and $\epsilon >0$, there exists $S > 0$ such that
$\| (\psi_S |_{E_R})^{-1} \colon \psi_S(E_R) \rightarrow E_R \|_{\rm cb} < 1 + \epsilon/2$.
It is easy to check that $(\psi_S |_{E_R})^{-1}$ is unital and self-adjoint. 
By Corollary B.9 of the book \cite{Ozawa:Book}, 
there exists a unital completely positive map
$\phi \colon B_S \rightarrow \mathbb{B}(\ell_2(X))$ which satisfies
$\| (\psi_S |_{E_R})^{-1} - \phi |_{\psi_S (E_R)} \|_{\rm cb} < \epsilon$.

Define a function $k$ on the set $X \times X$ by 
$k(y,z) = \langle \phi \circ \psi_S (e_{y,z}) \delta_z, \delta_y \rangle$,
where $e_{y, z}$ is the rank $1$ partial isometry which maps
$\delta_z$ to $\delta_y$.
Since $\phi \circ \psi_S$ is completely positive,
for every $x(1), x(2), \cdots, x(n) \in X$ and 
$\lambda_1, \lambda_2, \cdots, \lambda_n \in \mathbb{C}$,
we have
\begin{eqnarray*}
0
&\le&
\left\langle 
(\phi \circ \psi_S)^{(n)}
\left(
\left[
\begin{array}{cccc}
e_{x(1), x(1)} & \cdots & e_{x(1), x(n)} \\
\vdots & \ddots & \vdots \\
e_{x(n), x(1)} & \cdots & e_{x(n), x(n)} 
\end{array}
\right]
\right) 
\left[
\begin{array}{cccc}
\lambda_1 \delta_{x(1)}\\
\vdots\\
\lambda_n \delta_{x(n)}
\end{array}
\right] 
,
\left[
\begin{array}{cccc}
\lambda_1 \delta_{x(1)}\\
\vdots\\
\lambda_n \delta_{x(n)}
\end{array}
\right] 
\right\rangle\\
&=&
\sum_{i, j =1}^{n} \overline{\lambda_i} \lambda_j k(x(i), x(j)).
\end{eqnarray*}
It follows that 
$k$ is a positive definite kernel on $X$.
The kernel $k$ is supported on the set $\{(y, z) \in X^2 \ | \ d(y,z) \le S \}$, because
$\psi_S(e_{y,z}) = 0$ if $d(y, z) > S$.
In the case of $d(y, z) \le R$, we have 
\begin{eqnarray*}
\left| 1 - k(y, z) \right| 
= \left| \langle (e_{y,z} - \phi \circ \psi_S (e_{y,z})) \delta_z, \delta_y \rangle \right| 
\le \left\| ((\psi_S |_{E_R})^{-1} - \phi)(\psi_S (e_{y,z})) \right\| < \epsilon.
\end{eqnarray*}
It follows that $X$ satisfies the condition (\ref{DefOfA;Kernel}) in Proposition \ref{DefOfA}.
Hence $X$ has property A.
\end{proof}

The unital completely positive map $\phi$ 
in the first half of the proof 
was constructed 
in \cite[Proposition 11.41]{RoeLectureNote}.

\section{Conditions equivalent to property A}

In this revised version, 
we make comments on other coarse geometric properties.
By Theorem \ref{main} and 
a recent result \cite{BNSWW} by Brodzki, Niblo, \v{S}pakula, Willett, and Wright, 
we obtain the following theorem.
\begin{theorem}
For a metric space with bounded geometry, the following properties are equivalent:
\begin{enumerate}
\item\label{a}
property A,
\item\label{ula}
uniform local amenability (ULA$_\mu$) defined in \cite[Definition 2.5]{BNSWW},
\item\label{msp}
the metric sparsification property (MSP) defined in \cite[Definition 3.1]{ONLPoriginal},
\item\label{onl}
the operator norm localization property (ONL).
\end{enumerate}
\end{theorem}
The implication (\ref{a}) $\Rightarrow$ (\ref{ula}) is proved in \cite[Proposition 3.2]{BNSWW} 
and (\ref{ula}) $\Rightarrow$ (\ref{msp}) is proved in 
\cite[Proposition 3.8]{BNSWW}.
MSP implies ONL, which is shown in section 4 of
\cite{ONLPoriginal}.
With our result (\ref{onl}) $\Rightarrow$ (\ref{a}), we get the above theorem.

For a connected infinite graph $G$ of bounded vertex degrees, 
the notion of weighted hyperfiniteness was introduced by Elek and Tim\'{a}r \cite{Elek--Timar}.
Theorem \ref{wh} gives an answer to Question 1 in \cite{Elek--Timar}.
\begin{theorem}\label{wh}
The graph $G$ is weighted hyperfinite, if and only if its vertex set $V$ equipped with 
the graph metric $d$ has property A.
\end{theorem}
\begin{proof}
Weighted hyperfiniteness is invariant under 
quasi-isometry by Proposition 3.1 in \cite{Elek--Timar}.
Suppose that $G$ is weighted hyperfinite. Then the graph $G_R$ 
with the same vertex set $V$ and the edge set $\{(x, y) \in V \times V \ |\ d(x, y) < R\}$ 
is also weighted hyperfinite for every $R \ge 1$.
It is routine to prove that $G$ has MSP.
Conversely, MSP of $G$ implies weighted hyperfiniteness by definition.
It follows that weighted hyperfiniteness 
and MSP are equivalent.
\end{proof}
We thank Professor Ozawa for letting us know the paper \cite{Elek--Timar}.

\bibliographystyle{amsplain.bst}
\bibliography{onlp.bib}

\end{document}